\documentclass[a4paper]{article}

\usepackage[utf8]{inputenc}
\usepackage{subfigure}

\usepackage{lmodern}
\usepackage{geometry}
\usepackage{array}
\usepackage{graphicx}
\usepackage{amssymb, amsfonts, amsthm, amsmath}
\usepackage{enumerate}
\usepackage{cite}

\usepackage[english]{babel}
\usepackage{hyperref}
\usepackage{tikz}
\usetikzlibrary{decorations}

\numberwithin{equation}{section}

\theoremstyle{plain}
\newtheorem{theorem}{Theorem}[section]
\newtheorem{corollary}[theorem]{Corollary}
\newtheorem{lemma}[theorem]{Lemma}
\newtheorem{proposition}[theorem]{Proposition}

\usepackage{titlesec}

\titleformat{\section}{\normalfont\scshape\center}{\thesection}{1em}{}
\titleformat{\subsection}{\normalfont\scshape}{\thesubsection}{1em}{}
\titleformat{\subsubsection}{\normalfont\scshape}{\thesubsubsection}{1em}{}

\theoremstyle{definition}
\newtheorem{definition}[theorem]{Definition}

\theoremstyle{remark}
\newtheorem{remark}[theorem]{Remark}

\newcommand{\N}{\mathbb{N}}

\newcommand{\R}{\mathbb{R}}

\newcommand{\ind}[1]{\mathbf{1}_{\{#1\}}}

\DeclareMathOperator{\E}{\mathbf{E}}
\renewcommand{\P}{\mathbf{P}}
\DeclareMathOperator{\Var}{\mathbf{V}\mathrm{ar}}

\newcommand{\e}{\mathbf{e}}

\renewcommand{\bar}[1]{\overline{#1}}
\renewcommand{\tilde}[1]{\widetilde{#1}}
\renewcommand{\epsilon}{\varepsilon}
\renewcommand{\phi}{\varphi}
\renewcommand{\d}{\mathrm{d}}

\title{A $N$-branching random walk with random selection}
\author{Aser Cortines\thanks{Supported by the Israeli Science Foundation grant 1723/14} 
\ and \ Bastien Mallein\thanks{DMA, ENS}}
\date{\today}

\renewcommand{\e}{\mathrm{e}}
\newcommand{\dx}{\mathrm{d}x}

\newcommand{\eq}{\mathrm{eq}}
\newcommand{\egaldistr}{\overset{(d)}{=}}

\begin{document}

\maketitle

\begin{abstract}
We consider an exactly solvable model of branching random walk with random selection, which describes the evolution of a population with $N$ individuals on the real line. At each time step, every individual reproduces independently, and its offspring are positioned around its current locations. Among all children, $N$ individuals are sampled at random without replacement to form the next generation, such that an individual at position $x$ is chosen with probability proportional to $\e^{\beta x}$. We compute the asymptotic speed and the genealogical behavior of the system. 
\end{abstract}

\section{Introduction}
\label{sec:introduction}
In a general sense, a \emph{branching-selection particle system} is a Markovian process of particles on the real line evolving through the repeated application of the two steps:
\begin{description}
  \item[\emph{Branching step:}] every individual currently alive in the system splits into new particles, with positions (with respect to their birth place) given by independent copies of a point process.
  \item[\emph{Selection step:}] some of the new-born individuals are selected to reproduce at the next \emph{branching step}, while the other particles are ``killed".
\end{description}
We will often see the particles as individuals and their positions as their \emph{fitness}, that is, their score of adaptation to the environment.
From a biological perspective, branching-selection particle systems model the competition between individuals in an environment with limited resources. 

\begin{figure}[h]
\hspace{\stretch{1}}%
\subfigure[Branching step]{
\begin{tikzpicture}
  \draw [->] (0,0.5)--(0,-2);
  \draw (-0.5,0) node[left] {$t$} -- (4,0);
  \draw (-0.5,-1.5) node[left] {$t+1$} -- (4,-1.5);

  \draw [thick, color=magenta] (0.8,0) node{$\bullet$};
  \draw [thick, color=magenta] (3,-1.5) node{$\bullet$} -- (0.8,0) -- (0.9,-1.5) node{$\bullet$};
  \draw [thick, color=magenta] (0.2,-1.5) node{$\bullet$} -- (0.8,0);

  \draw [thick, color=green] (1.2,0) node{$\bullet$};
  \draw [thick, color=green] (2.5,-1.5) node{$\bullet$} -- (1.2,0) -- (1.4,-1.5) node{$\bullet$};
  \draw [thick, color=green] (0.4,-1.5) node{$\bullet$} -- (1.2,0);

  \draw [thick, color=red] (1.9,0) node{$\bullet$};
  \draw [thick, color=red] (3.2,-1.5) node{$\bullet$} -- (1.9,0) -- (2.2,-1.5) node{$\bullet$};
  \draw [thick, color=red] (1.8,-1.5) node{$\bullet$} -- (1.9,0);

  \draw [thick, color=blue] (2.8,0) node{$\bullet$};
  \draw [thick, color=blue] (3.7,-1.5) node{$\bullet$} -- (2.8,0) -- (2.6,-1.5) node{$\bullet$};
  \draw [thick, color=blue] (1.6,-1.5) node{$\bullet$} -- (2.8,0);
\end{tikzpicture}
}%
\hspace{\stretch{1}}%
\subfigure[Selection step]{
\begin{tikzpicture}
  \draw [densely dashed, color=magenta!30] (0.2,-1.5) node{$\bullet$} -- (0.8,0) -- (0.9,-1.5) node{$\bullet$};
  \draw [densely dashed, color=green!30] (2.5,-1.5) node{$\bullet$} -- (1.2,0) -- (1.4,-1.5) node{$\bullet$};
  \draw [densely dashed, color=green!30] (0.4,-1.5) node{$\bullet$} -- (1.2,0);
  \draw [densely dashed, color=red!30] (1.8,-1.5) node{$\bullet$} -- (1.9,0) -- (2.2,-1.5) node{$\bullet$};
  \draw [densely dashed, color=blue!30] (3.7,-1.5) node{$\bullet$} -- (2.8,0);

  \draw [->] (0,0.5)--(0,-2);
  \draw (-0.5,0) node[left] {$t$} -- (4,0);
  \draw (-0.5,-1.5) node[left] {$t+1$} -- (4,-1.5);

  \draw [thick, color=magenta] (0.8,0) node{$\bullet$};
  \draw [thick, color=magenta] (3,-1.5) node{$\bullet$} -- (0.8,0);
  \draw [thick, color=green] (1.2,0) node{$\bullet$};
  \draw [thick, color=red] (1.9,0) node{$\bullet$};
  \draw [thick, color=red] (3.2,-1.5) node{$\bullet$} -- (1.9,0);  
  \draw [thick, color=blue] (2.8,0) node{$\bullet$};
  \draw [very thick, color=blue] (1.6,-1.5) node{$\bullet$} -- (2.8,0) -- (2.6,-1.5) node{$\bullet$};
\end{tikzpicture}
}%
\hspace{\stretch{1}}%
\caption{One time step of a branching-selection particle system}
\end{figure}
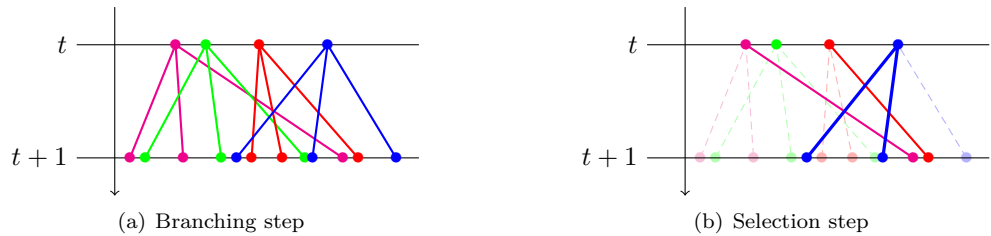

These models are of physical interest \cite{BD97,BDMM07} and can be related to reaction-diffusion phenomena and the F-KPP equation. Different methods can be used to select the individuals. For example, one can consider an \emph{absorbing barrier}, below which particles are killed \cite{AiJ11,BBS13,Mal15a,Pai16}. Another example is the case where only the $N$ rightmost individuals are chosen to survive \cite{BD97,BDMM07,BeG,DuR}, the so-called ``$N$-branching random walk". In this paper, we introduce a new selection mechanism, in which the individuals are randomly selected with probability depending on their positions.

Based on numerical simulations \cite{BD97} and the study of solvable models \cite{BDMM07}, it has been predicted that the dynamical and structural aspects of many branching selection particle systems satisfy universal properties. For example, The cloud of particles travels at speed $v_N$, which converges to a limit $v$ as the size of the population $N$ diverges. It has been conjectured \cite{BDMM07} that 
\[
v_N -v = - \phi \left(\log N +3\log \log N + o(\log \log N) \right)^{-2} 
\quad \text{as $N \to \infty$,}
\] 
for an explicit constant $\phi$ depending on the law of reproduction. 

Some of these conjectures have been recently proved for the $N$-branching random walk \cite{BeM,BeG,CoG, DuR,Mai,Mal15b}. B\'erard and Gou\'er\'e \cite{BeG} prove that $v_N -v$ behaves like $-\phi(\log N)^{-2}$. Nevertheless, several conjectures about this process remain open, such as the asymptotic behavior of the genealogy or the second-order expansion of the speed. 
Other examples in which the finite-size correction to the speed of a branching-selection particle system is explicitly computed can be found in \cite{BeM,CoC15,CoG,DuR,Mal15b,MMQ11}. 

To study the genealogical structure of such models we define the \emph{ancestral partition process} $\Pi^N_n(t)$ of a population choosing $n \ll N$ individuals from a given generation $T$ and tracing back their genealogical linages. That is, $\Pi^N_n(t)$ is a process in $\mathcal{P}_n$ the set of partitions (or equivalence classes) of $[n] := \{ 1,\ldots, n \}$ such that $i$ and $j$ belong to the same equivalence class if the individuals $i$ and $j$ have a common ancestor $t$ generations backwards in time. Notice that the direction of time is the opposite of the direction of time for the natural evolution of the population, that is, $t=0$ is the current generation, $t=1$ brings us one generation backward in time and so on. 

It has also been conjectured \cite{BDMM07} that the genealogical trees of branching selection particle systems converge to those of a Bolthausen-Sznitman coalescent and that the average coalescence times scale like a power of the logarithm of the population size. These conjectures contrast with classical results in neutral population models, such as Wright-Fisher and Moran's models, that lay in the Kingman coalescent universality class \cite{MoS01}. Mathematically, these conjectures are difficult to be verified and they have only been proved for some particular models \cite{BBS13,Cor14}.

We define in this article a solvable model of branching selection particle system evolving in discrete time, and compute its asymptotic speed as well as its genealogical structure. Given $N \in \N$ and $\beta > 1$, it consists in a population with a fixed number $N$ of individuals. At each time step, the individuals die giving birth to offspring that are positioned according to independent Poisson point processes with intensity $\e^{-x} \mathrm{d}x$ (that we write PPP($\e^{-x} \mathrm{d}x$) for short). Then, $N$ individuals are sampled (without replacement) to form the next generation, such that a child at the position $x$ is sampled with probability proportional to $\e^{\beta x}$. 

To describe the model we introduce the following notation. Let $X^N_0(1),\ldots, X^N_0(N) \in \R$ be the initial position of the particles and $\{ \mathcal{P}_t(j), j \leq N, t \in \N \}$ be a family of i.i.d. PPP($\e^{-x} \mathrm{d}x$). Given $t \geq 1$ and $X^N_{t-1}(1),\ldots, X^N_{t-1}(N)$ the $N$ positions at time $t-1$, we define the new positions as follows:
\begin{enumerate}[i.]
  \item Each individual $X_{t-1}^N(j)$ gives birth to infinitely many children that are positioned according to the point process $X^N_{t-1}(j) + \mathcal{P}_t(j)$. 
  Let $\Delta_t :=(\Delta_t(k) ;\: k \in \N)$ be the sequence obtained by all positions ranked decreasingly, that is
  \[
    \left( \Delta_t(k), k \in \N \right) = \mathrm{R}\mathrm{ank} \left(\left\{ X^N_{t-1}(j) + p ;\: p \in \mathcal{P}_t(j), j \leq N \right\} \right).
  \]
  \item We sample successively $N$ individuals $X_{t}^N(1), \ldots, X_{t}^N(N)$ composing the $t$th generation from $\{\Delta_t(1), \Delta_t(2), \ldots\}$ such that for all $i \in \{1, 2, \ldots, N\}$:
\begin{equation}\label{equa.def.N,b,BRW}
\P\left( \left. X^N_t(i) = \Delta_t(j) \right| \Delta_t , X^N_t(1),\ldots, X^N_t(i-1)\right)
 =  \frac{\e^{\beta \Delta_t(j)} \ind{\Delta_t(j) \not \in \{ X^N_t(1),\ldots X^N_t(i-1)\}}}
{\sum_{k=1}^{+\infty} \e^{\beta \Delta_t(k)} - \sum_{k=1}^{i-1} \e^{\beta X^N_t(k)}}.
\end{equation}
\end{enumerate}
To keep track of the genealogy of the process we define 
  \begin{equation}
    \label{eqn:definitionAN}
    A^N_t(i) = j \quad \text{if} \quad X^N_t(i) \in \left\{ X^N_{t-1}(j) + p, p \in \mathcal{P}_t(j)\right\},
  \end{equation}
that is, $A_n(i)=j$ if $X^N_t(i)$ is an offspring of $X^N_{t-1}(j)$. We call this system the $(N,\beta)$-\emph{branching random walk} or $(N,\beta)$-BRW for short. 

It can be checked that the sum in the denominator of \eqref{equa.def.N,b,BRW} is finite if $\beta \in (1,\infty)$ and that it diverges as $\beta \to 1$ (see Proposition~\ref{prop:def} below), thus the model is only defined for $\beta \in (1,\infty)$. Notice that as $\beta \to \infty$ the sum in the denominator is dominated by the high values of 
$\Delta_t$. Precisely, it can be checked that the following limits hold a.s.
\[
\textstyle
\lim_{\beta \to \infty} \e^{-\beta \Delta_t(1)}\sum_{k=1}^{+\infty} \e^{\beta \Delta_t(k)} = 1 ,
\quad
\lim_{\beta \to \infty} \e^{-\beta \Delta_t(2)}\sum_{k=2}^{+\infty} \e^{\beta \Delta_t(k)} = 1 ,
\quad \text{and so on.}
\] 
Therefore, the case ``$\beta = \infty$" is the ``exponential model" from \cite{BD97,BD12,BDMM07}, in which the $N$ rightmost individuals are selected to form the next generation. In contrast with the examples already treated in the literature, when $\beta < \infty$ one does not necessarily select the rightmost offspring. In this paper, we will take interest in the dynamical and genealogical aspects of the $(N,\beta)$-BRW, showing that it travels at a deterministic speed and that its genealogical trees converge in distribution. The next result concerns the speed of the $(N,\beta)$-BRW.

\begin{theorem}
\label{thm:main_vel}
For all $N \in \N$ and $\beta \in (1,\infty]$, there exists $v_{N,\beta}$ such that
\begin{equation}
\label{equa:lem:speed} 
\lim_{t \to +\infty} \frac{\max_{j \leq N} X^N_t(j)}{t} =\lim_{n \to +\infty} \frac{\min_{j \leq N} X^N_t(j)}{t} = v_{N,\beta} \qquad a.s.
\end{equation}
moreover, $v_{N,\beta} = \log \log N + o(1)$ as $N\to \infty$.
\end{theorem}
 
The main result of the paper is the following theorem concerning the convergence in law of the ancestral partition process $\big(\Pi^N_n(t); t \in \N \big)$ of the $(N,\beta)$-BRW.

\begin{theorem}
\label{thm:main_coal}
For all $N \in \N$ and $\beta \in (1,\infty]$, let $c_N$ be the probability that two individuals uniformly chosen at random have a common ancestor one generation backwards in time. Then, we have $ \lim_{N \to \infty} c_N \log N =1$ and the rescaled coalescent process $\big(\Pi^N( \lfloor t/c_N \rfloor ), t \geq 0 \big)$ converges in distribution toward the \emph{Bolthausen-Sznitman coalescent}.
\end{theorem}

The Bolthausen-Sznitman coalescent in Theorem~\ref{thm:main_coal} can be roughly explained by an individual going far ahead of the rest of the population, so that its offspring are more likely to be selected and overrun the next generation. Based on precise asymptotic of the coalescence time, the authors in \cite{BDMM07} argue that the genealogical trees of the exponential model converge to the Bolthausen-Sznitman coalescent and conjecture that this behavior should be expected for a large class of models. The $(N,\beta)$-BRW can be though as a finite temperature version of the exponential model from \cite{BDMM07}. In this sense, Theorem~\ref{thm:main_coal} attests for the robustness of their conjectures showing that even under weaker selection constrains this convergence occurs. It indicates that whenever the rightmost particles are likely to be selected, then the Bolthausen-Sznitman coalescent is to be expected.

Different coalescent behavior should be expected when the selection mechanism does not favor the rightmost particles, the classical example being the Wright-Fisher model. Another example can be obtained modifying the selection mechanism of the $(N,\beta)$-BRW. It can be checked using the techniques developed in this paper (see for example Theorem \ref{thm:mainbege}) that if we systematically eliminate the first individual sampled $X_1^N(t)$, so that it does not reproduce in the next generation, then this new branching-selection particle system lays in the Kingman's coalescent universality class. Notice that this new selection procedure no longer favors the rightmost particles (in this case, the rightmost particle), which justifies this change of behavior.

\paragraph*{Notation.} In this article, we write 
\begin{multline*}
 f(x) \sim g(x) \quad \text{as $x \to a$ if} \quad \lim_{x \to a} \frac{f(x)}{g(x)}=1;
 \qquad f(x) = o(g(x)) \text{ as $x \to a$ if} \quad \lim_{x \to a} \frac{f(x)}{g(x)}=0;\\
 \text{and} \qquad 
 f(x) = \mathcal{O}(g(x)) \text{ as $x \to a$ if} \quad \limsup_{x \to a} \frac{f(x)}{g(x)} < +\infty.\\
 \end{multline*}

\subsection*{Preliminary results}
\label{sec:generalproperties}
In this section, we prove that the $(N,\beta)$-BRW is well defined and provide some elementary properties such as the existence of the speed $v_{N,\beta}$.

\begin{proposition}
\label{prop:def}
The $(N,\beta)$-BRW is well-defined for all $N \in \N$ and $\beta\in (1,\infty]$.
Moreover, setting $X^N_t(\eq) := \log \sum_{j=1}^N \e^{X^N_t(j)}$, 
the sequence $\big(\sum_{k \in \N} \delta_{\Delta_k(t+1)-X^N_t(\eq)} :\: t \in \N \big)$ is an i.i.d. family of Poisson point processes with intensity measure $\e^{-x} \mathrm{d}x$.
\end{proposition}

\begin{proof}
With $N $ and $\beta$ fixed, assume that the process has been constructed up to time $t$ with $X^N_t(1), \ldots, X^N_t(N)$ denoting the positions of the $N$ particles. Thanks to the invariance of superposition of independent PPP, $\big\{ X^N_{t}(j) + p ;\: p \in \mathcal{P}_t(j), j \leq N \big\}$ is also a PPP with intensity~measure 
\[\textstyle
\sum_{i=1}^N \e^{-(x-X_t(i))}\dx = \e^{-(x-X^N_t(\eq))}\dx.
\]
Therefore, with probability one: all points have multiplicity one, the sequence $(\Delta_k(t+1);\; k \in \N)$ is uniquely defined. Since there are finitely many points $\Delta_k(t+1)$ that are positive and $\E \big( \sum \e^{\beta \Delta_k(t+1)} \mathbf{1}_{\{ \Delta_k(t+1) <0 \}} \big) < \infty$ we have that $\sum \e^{\beta \Delta_k(t+1)} < \infty$ a.s. As a consequence, the selection step is well-defined, proving the first claim. Moreover, $(\Delta_k(t+1)-X_t(\eq);\; k \in \N)$ is a PPP($\e^{-x}\dx$) independent from the $t$ first steps of the $(N,\beta)$-BRW, proving the second claim.
\end{proof}

\begin{remark}
It is convenient to think $X^N_t(\eq)$ as an ``equivalent position'' of the front at time $t$, in the sense that the particles positions in the $(t+1)$th generation are distributed as if they were generated by a unique individual positioned at $X^N_t(\eq)$.
\end{remark}

We use Proposition~\ref{prop:def} to prove the existence of the \emph{speed} $v_{N,\beta}$, the study of its asymptotic behavior is postponed to Section~\ref{sec:piy}.

\begin{lemma}
\label{lem:speed}
With the notation of the previous proposition, \eqref{equa:lem:speed} in Theorem~\ref{thm:main_vel} holds with
\[v_{N,\beta} := \E(X^N_1(\eq)-X^N_0(\eq)).\]
\end{lemma}

\begin{proof}
By Proposition~\ref{prop:def}, $\big(X^N_{t+1}(\eq)- X^N_{t}(\eq) :\: t \in \N \big)$ are i.i.d. random variables with finite mean, therefore 
\[
 \lim_{t \to +\infty} \frac{X^N_t(\eq)}{t} = v_{N,\beta} 
 \quad \text{a.s. by the law of large numbers.}
\]
Notice that both $\big(\max X^N_t(j)-X^N_{t-1}(\eq) \big)_t$ and $\big(\min X^N_t(j)-X^N_{t-1}(\eq)\big)_t$ are sequences of i.i.d. random variables with finite mean, which yields (\ref{equa:lem:speed}).
\end{proof}

In a similar way, we are able to obtain a simple structure for the genealogy of the process, and describe its law conditionally on the position of the particles.
\begin{lemma}
\label{lem:coal}
The sequence $(A^N_t)_{t \in \N}$ defined in \eqref{eqn:definitionAN} is i.i.d. Moreover, it remains independent conditionally on $\mathcal{H} = \sigma(X^N_t(j), j \leq N, t \geq 0)$, with the conditional probabilities
\begin{align}\label{equa:lem:coal_coeficient}
\P\big(A^N_{t+1}= \bar{k} \mid \mathcal{H} \big) = \theta^N_t(k_1) \! \ldots \! \theta^N_t(k_N), \quad & \text{where} \quad 
 \bar{k} = (k_1 , \ldots , k_N) \in \{1, \ldots ,N\}^{N}; \notag \\
 & \text{and} \quad
 \theta^N_t(k) := \frac{\e^{X^N_t(k)}}{\sum_{i=1}^N \e^{X^N_t(i)}} .
\end{align}
\end{lemma}

\begin{proof}
To each point $x \in \sum \delta_{\Delta_{t+1}(k) - X^N_t(\eq)}$ we associate the mark $i$ if it is a point coming from $X^N_t(i) + \mathcal{P}_t(i)$. The invariance under superposition of independent PPP says that 
\[
\P\big(x \in X^N_t(i) + \mathcal{P}_t(i)\mid \mathcal{H} \big) = \frac{\e^{-(x- X^N_t(i))}}{ \sum_{j=1}^N \e^{-(x- X^N_t(j))}} = \frac{\e^{X^N_t(i)}}{ \sum_{j=1}^N \e^{X^N_t(j)}}.
\] 
By definition of $A_t^N(i)$, it is precisely the mark of $X^N_{t+1}(i)$, which yields (\ref{equa:lem:coal_coeficient}). The independence between the $A_t^N$ can be easily checked using Proposition~\ref{prop:def}. 
\end{proof}

\paragraph*{Organization of the paper.} In Section~\ref{sec:pddistribution}, we obtain some technical lemmas concerning the Poisson-Dirichlet distributions. We focus in Section~\ref{sec:cvexchangeable} on a class of coalescent processes generated by Poisson Dirichlet distributions and we prove a convergence criterion. Finally, in Section~\ref{sec:piy}, we provide an alternative construction of the $(N,\beta)$-BRW in terms of a Poisson-Dirichlet distribution, and we use the results obtained in the previous sections to prove Theorems~\ref{thm:main_vel} and~\ref{thm:main_coal}.
\section{Poisson-Dirichlet distribution}
\label{sec:pddistribution}
In this section, we focus on the two-parameter Poisson-Dirichlet distribution denoted as PD$(\alpha,\theta)$ distribution.

\begin{definition}[Definition~1 in \cite{PiY97}] \label{PD_def} For $\alpha \in (0,1)$ and $\theta > -\alpha$, let $(Y_j :\: j \in \N)$ be a family of independent r.v. such that $Y_j$ has Beta$(1-\alpha, \theta +j\alpha )$ distribution and write
\[
V_1 = Y_1, 
\quad \text{and} 
\quad V_j = \prod_{i=1}^{j-1} (1-Y_i) Y_j, \qquad \text{if } j \geq 2.
\]
Let $U_1 \geq U_2 \geq \cdots$ be the ranked values of $(V_n)$, we say that the sequence $(U_n)$ is the \emph{Poisson-Dirichlet distribution with parameters} $(\alpha,\theta)$. 
\end{definition}

Notice that for any $k \in \N$ and $n \in \N$, we have
\[
 \P\big( V_n = U_k \mid (U_j, j \in \N), V_1,\ldots, V_{n-1} \big) = \frac{U_k \ind{U_k \not \in \{ V_1,\ldots V_{n-1}\}}}{1 - V_1 -V_2 - \cdots - V_{n-1}},
\]
for this reason we say that $(V_n)$ follows the \emph{size-biased pick} from a PD$(\alpha,\theta)$. It is well known that there exists a strong connexion between PD distributions and PPP \cite{PiY97}, we recall some of these results in the proposition below.

\begin{proposition}[Proposition 10 in \cite{PiY97}]
\label{prop:PD-PPP}
Let $x_1 > x_2 > \ldots $ be the points of a $\mathrm{PPP}(\e^{-x} \mathrm{d}x)$ and write 
$L= \sum_{j=1}^{+\infty} \e^{\beta x_j}$ and $U_j = \e^{\beta x_j} / L$. Then $(U_j, j \geq 1)$ has $\mathrm{PD}(\beta^{-1},0)$ distribution and
\[
 \lim_{n \to +\infty} n^\beta U_n = 1/L \quad \text{a.s.}
\]
\end{proposition}

Notice from Propositions~\ref{prop:def} and~\ref{prop:PD-PPP} that $(\e^{\beta X_n^N(i)}/\sum_{j=1}^{+\infty} \e^{\beta \Delta_n(j)})$ has the distribution of $(V_i)$ the size-biased pick from $\mathrm{PD}(\beta^{-1},0)$, which makes the model solvable. 

\begin{remark}[Change of parameter] \label{rem:changeofparameter} If $V_1, V_2, \ldots$ is a \emph{size-biased pick from} a PD$(\alpha,\theta)$, then 
\[
\frac{V_2}{1-V_1}, \frac{V_3}{1-V_1}, \ldots, 
\quad
\text{ has the distribution of a size-biased pick from a PD$(\alpha,\alpha+\theta)$,}
\] 
moreover, it is independent of $V_1$. That is, the sequence obtained from $V_1, V_2, \ldots$ after discarding the first sampled element $V_1$ and re-normalizing is a \emph{size-biased pick from} a PD$(\alpha, \alpha+ \theta)$. Therefore, ordering this sequence one obtains a PD$(\alpha, \alpha+ \theta)$ sequence. 
\end{remark}

In what follows, we fix $\alpha \in (0,1)$ and $\theta > -\alpha$, and let $c$ and $C$ be positive constants, that may change from line to line and implicitly depend on $\alpha$ and $\theta$. We will focus attention on the convergence and the concentration properties of 
\begin{equation}
 \label{eqn:defSigmaN}
 \Sigma_n := \sum_{j=1}^n V_j^{\alpha} = \sum_{j=1}^n Y_j^{\alpha} \prod_{i=1}^{j-1}(1-Y_i)^{\alpha}; \qquad n \in \N.
\end{equation}

\begin{lemma}
\label{lem:martingale}
Set $M_n := \prod_{i=1}^{n}(1-Y_i)$, then, there exists a positive r.v. $M_\infty$ such that
\begin{equation}
 \label{eqn:pconvergence}
 \lim_{n \to +\infty} \left(n^{\frac{1-\alpha}{\alpha}} M_n\right)^\gamma = M^\gamma_\infty \quad \text{a.s. and in } \mathbb{L}^1 \text{ for all } \gamma > -(\theta+\alpha),
\end{equation}
with $\gamma$-moment verifying $\E(M_\infty^\gamma) = \Phi_{\theta,\alpha}(\gamma):= \alpha^\gamma\frac{\Gamma(\theta+1)\Gamma\big(\frac{\theta+\gamma}{\alpha} +1 \big)}{\Gamma(\theta + \gamma+1) \Gamma\big( \frac{\theta}{\alpha}+1 \big)}$.
Moreover, if $0 < \gamma < \theta+\alpha$, then there exists $C_\gamma > 0$ such that
\begin{equation}
 \label{eqn:pupperbound}
 \P\left( \inf_{n \geq 0} n^{\frac{1-\alpha}{\alpha}}M_n \leq y \right)\leq C_\gamma y^\gamma, \qquad \text{for all $n \geq 1$ and $y \geq 0$}.
\end{equation}
\end{lemma}

Notice that if $\gamma > -\theta$, then $\Phi_{\theta,\alpha}(\gamma)= \alpha^\gamma\frac{\Gamma(\theta)\Gamma\big(\frac{\theta+\gamma}{\alpha} \big)}{\Gamma(\theta + \gamma) \Gamma\big( \frac{\theta}{\alpha} \big)}$.

\begin{proof}
Fix $\gamma > -(\theta + \alpha)$, then $\big(M_n^\gamma /\E(M_n^\gamma)\big)$ is a non-negative martingale with respect to its natural filtration and 
\begin{align*}
 \E\left( M_n^\gamma \right) &= 
 \frac{\Gamma(\theta + \gamma + n\alpha )}{\Gamma(\theta + n\alpha)}
 \frac{\Gamma\big(n+ \frac{\theta}{\alpha}\big)}{\Gamma \big(n + \frac{\theta+\gamma}{\alpha} \big)} 
 \frac{\Gamma(\theta+1) \Gamma\big( \frac{\theta + \gamma}{\alpha} +1\big)}{\Gamma(\theta + \gamma+1)\Gamma\big(\frac{\theta}{\alpha} +1 \big)} \\
 &\sim \Phi_{\theta,\alpha}(\gamma) n^{-\gamma \frac{1-\alpha}{\alpha}}, \quad \text{as } n \to +\infty.
\end{align*}
Since $\lim_{n \to +\infty} \E(M_n^{\gamma/2})\E(M_n^\gamma)^{-1/2} > 0$, Kakutani's theorem says that $M_n^\gamma/\E(M_n^\gamma)$ converges a.s. and in $\mathbb{L}^1$ as $n \to +\infty$, implying
\eqref{eqn:pconvergence} with $M_\infty = \lim_{n \to +\infty} M_n n^{\frac{1-\alpha}{\alpha}}$. 

In particular, if $0 <\gamma < \theta + \alpha$ we obtain from Doob's martingale inequality that 
\[
 \P\left( \inf_{n \geq 0} n^{\frac{1-\alpha}{\alpha}}M_n \leq y \right) = \P\left( \sup_{n \geq 0} M_n^{-\gamma}n^{-\gamma\frac{1-\alpha}{\alpha}} \geq y^{-\gamma} \right) \leq \P\left( \sup_{n \geq 0} \tfrac{M_n^{-\gamma}}{\E(M_n^{-\gamma})} \geq y^{-\gamma}/C_\gamma \right),
\]
with $C_\gamma := \sup_{n\in \N} \E[M_n^{-\gamma}]/ n^{\gamma\frac{1-\alpha}{\alpha}} < \infty $, proving \eqref{eqn:pupperbound}.
\end{proof}

We now focus on the convergence of the series 
$\sum Y_j^{\alpha} j^{\alpha-1}$.

\begin{lemma}
\label{lem:sumindep}
Let $S_n := \sum_{j=1}^n Y_j^{\alpha} j^{\alpha-1}$ and $\Psi_\alpha := \alpha^{-\alpha} \Gamma(1-\alpha)^{-1}$, then, there exists a random variable $S_\infty$ such that 
\[
 \lim_{n \to +\infty} S_n - \Psi_\alpha \log n = S_\infty \quad \text{a.s.}
\]
Moreover, there exists $C>0$ such that for all $n \in \N$ and $y \geq 0$,
\[
 \P\left( |S_n-\E(S_n)| \geq y \right) \leq C \e^{-y^{2-\alpha}}.
\]
\end{lemma}

\begin{proof}
Since $Y_j$ has Beta$(1-\alpha, \theta+j\alpha)$ distribution, we have
\[
\E((jY_j)^\alpha) = \frac{1}{\alpha^\alpha \Gamma(1-\alpha)} + \mathcal{O}(1/j) 
\quad \text{and} \quad 
\mathbf{V}\mathrm{ar}((jY_j)^\alpha) = \frac{\Gamma(1+\alpha)\Gamma(1-\alpha) - 1 }{\alpha^{2\alpha} \Gamma(1-\alpha)^2} + \mathcal{O}(1/j), 
\]
which implies that 
$\sum \Var(Y_j^{\alpha-1}) < +\infty$ 
and that $\E(S_n) =\Psi_\alpha \log n + C_S + o(1)$ 
with $C_S \in \R$. Thanks to 
$Y_j^\alpha - \E(Y_j^\alpha) \in (-1,1)$ a.s.
we deduce from Kolmogorov's three-series theorem that $S_n-\E(S_n)$ and hence that $S_n - \Psi_\alpha \log n$ converge a.s. 
To bound $\P(S_n - \E(S_n) \geq y)$, notice that 
\[
 \P(S_n - \E(S_n) \geq y) \leq \e^{-\lambda y} \E\left[ \e^{\lambda(S_n - \E(S_n))} \right] \leq \e^{-\lambda y} \prod_{j=1}^n \E\left( \e^{\lambda j^{\alpha - 1} (Y_j^\alpha - \E(Y_j^\alpha))} \right),
\]
for all $y \geq 0$ and $\lambda > 0$.
Taking $c>0$ such that $\e^x \leq 1 + x + c x^2$ for $x \in (-1,1)$, we obtain
\[
 \E\left[ \e^{\lambda j^{\alpha - 1} (Y_j^\alpha - \E(Y_j^\alpha))} \right]
 \leq 
 \begin{cases}
  \e^{\lambda j^{\alpha - 1}} & \mathrm{if} \quad \lambda j^{\alpha-1} > 1; \\
  1 + c \lambda^2 j^{2(\alpha - 1)} \mathbf{V}\mathrm{ar}(Y_j^\alpha) & \mathrm{if} \quad \lambda j^{\alpha-1} \leq 1.
 \end{cases}
\]
Since $\sum_{j^{1-\alpha} \leq \lambda} j^{\alpha - 1} < \lambda^{\frac{1}{1-\alpha}}$ for all $\alpha \in (0,1)$ and $\theta \geq 0$, there exists $c= c(\alpha, \theta)$ such that 
\begin{align*}
 \P\left( S_n - \E(S_n) \geq y \right) &\leq \e^{-\lambda y} \prod_{j^{1-\alpha} \leq \lambda} \e^{\lambda j^{\alpha - 1}} 
 \times \prod_{j^{1-\alpha} > \lambda} \left( 1 + c \frac{\lambda^2}{j^{2}} \right)\\
 &\leq \exp \left( -\lambda y +\lambda^{\frac{2-\alpha}{1-\alpha}} +c \lambda^2 \right),
 \quad \text{ for all $n \in \N$ and $y \geq 0$}.
\end{align*}
Let $\varrho := (2-\alpha)/(1-\alpha)> 2$, then there exists $C = C(\alpha,\theta) >0$ such that 
\[
 \P\left( S_n - \E(S_n) \geq y \right) \leq C\exp \left( -\lambda y + C \lambda^{\varrho }\right),
 \quad \text{for all $n \in \N$ and $y \geq 0$}.
\]
Optimizing in $\lambda > 0$ we obtain
\[
 \P\left( S_n - \E(S_n) \geq y \right) \leq C \exp\left(-y^{\varrho/(\varrho-1)}C^{1/(1-\varrho)} \left(\varrho^{1/(1-\varrho)} - \varrho^{\varrho/(1-\varrho)} \right) \right),
\]
with 
$C^{1/(1-\varrho)} \left[\varrho^{1/(1-\varrho)} - \varrho^{\varrho/(1-\varrho)} \right]> 0$, as $\varrho > 1$.
The same argument, with the obvious changes, holds for $\P(S_n - \E(S_n) \leq -y)$, therefore, there exists $C>0$ such that
$
\P\left( \left| S_n - \E(S_n) \right| \geq y \right) \leq C \exp\left(-y^{\varrho/(\varrho-1)}/C\right),
$
proving the second statement.
\end{proof}

With the above results, we obtain the convergence of $\Sigma_n= \sum V_j^\alpha$ as well as its tail probabilities.

\begin{lemma}
\label{lem:mainbege}
With the notation of Lemmas \ref{lem:martingale} and \ref{lem:sumindep}, we have
\[
 \lim_{n \to +\infty} \frac{\Sigma_n}{\log n} = \Psi_\alpha M_\infty^\alpha \quad \text{a.s. and in } \mathbb{L}^1.
\]
Moreover, for any $0 < \gamma < \alpha + \theta$ there exists $D_{\gamma}$ such that for all $n \geq 1$ large enough and $u > 0$
\[
 \P\left( \Sigma_n \leq u \log n \right) \leq D_{\gamma} u^{\frac{\gamma}{\alpha}}.
\]
\end{lemma} 

\begin{proof}
Notice that $\Sigma_n = \sum_{j=1}^n (S_j - S_{j-1}) j^{1-\alpha}M^\alpha_{j-1}$ (since $S_0 :=0$) and that 
\[
 \lim_{n \to +\infty} \big( M_n (n+1)^{\frac{1-\alpha}{\alpha}} \big)^\alpha = M_\infty^\alpha \quad \text{and} \quad \lim_{n \to +\infty} \frac{S_n}{\log n} = \Psi_\alpha \quad \text{a.s.}
\]
by Lemmas~\ref{lem:martingale} and \ref{lem:sumindep} respectively. 
Then, Stolz-Cesàro theorem yields $ \Sigma_n/\log n \to \Psi_\alpha M_\infty^\alpha$ a.s. as $n\to \infty$.
Expanding $\big(\Sigma_n\big)^2$ we obtain
\begin{align*}
 \E\left( \Sigma_n^2 \right)
 = & \sum_{j=1}^n \E\left(V_j^{2\alpha}\right) + 2 \sum_{i=1}^{n-1} \sum_{j=i+1}^n \E\left( (V_i V_j)^\alpha \right) \\
 = & \sum_{j=1}^n \E\left( M_{j-1}^{2\alpha} \right) \E(Y_j^{2\alpha}) + 2 \sum_{i=1}^{n-1} \E\left( M_{i-1}^{2\alpha} \right) \E\left( (Y_i(1-Y_i))^\alpha \right) \sum_{j=i+1}^n \E\left( \tfrac{M_{j-1}^\alpha}{M_i^\alpha} \right) \E(Y_j^\alpha)\\
 \leq & C \sum_{j=1}^n j^{-2(1-\alpha)} j^{-2\alpha} + C \sum_{i=1}^{n-1} i^{-2(1-\alpha)} i^{-\alpha} \sum_{j=i+1}^n \frac{j^{-(1-\alpha)}}{i^{-(1-\alpha)}} j^{-\alpha} \leq C (\log n)^2.
\end{align*}
Therefore, $\sup\E\big[ ( \Sigma_n/\log n )^2 \big] < +\infty$ implying its $\mathbb{L}^1$ convergence. To obtain bounds for $\P(\Sigma_n \leq u \log n)$, we study the two cases $u \geq 1/n$ and $u\leq 1/n$ separately. 
Assume first that $u\geq 1/n$, then 
\[
 \Sigma_n = \sum_{j=1}^n \left((j-1)^\frac{1-\alpha}{\alpha} M_{j-1}\right)^\alpha \frac{(jY_j)^\alpha}{j} \geq \left( \inf_{j \in \N} j^\frac{1-\alpha}{\alpha} M_j\right)^\alpha S_n.
\]
For all $\gamma' < \theta + \alpha$ and $t > 0$ such that $t < \E[S_n]$ we have
\begin{align*}
 \P\left(\Sigma_n \leq u \log n \right)
 &\leq 
 \P\left( S_n \leq t \right) 
 + \P\left( \left(\inf j^{\frac{1-\alpha}{\alpha}} M_j\right)^\alpha \leq (u \log n)/t \right)\\
 &\leq C \exp\left(-C^{-1}(\E[S_n] - t)^{\frac{\varrho}{(\varrho-1)}}\right) 
 + C_{\gamma'} \left(\tfrac{u \log n}{t}\right)^{\gamma'/\alpha}.
\end{align*}
Let $0<\varepsilon < 1/2$ and set $t= u^\varepsilon \log n$, since $\lim_{n \to +\infty} \frac{E(S_n)}{\log n} = \Psi_\alpha$, there exists a constant $c > 0$ depending only on $\alpha$ such that $u^\varepsilon \log n \leq E[S_n]$ for all $u \leq c$. Decreasing $c$ if necessary, we can and will assume that $C^{-1} (\E[S_n] - y^\varepsilon \log n) < a \log n$ for all $\varepsilon < 1/2$ and hence that 
\[
 \P\left(\Sigma_n \leq u \log n \right) \leq C \exp\left(-(a \log n)^{\varrho/(\varrho-1)}\right) + C_{\gamma'} u^{(1-\varepsilon)\gamma'/\alpha}, \quad \text{ for all } u \leq c,
\] 
where $a > 0$ is to be chosen conveniently small. Observe that
\[
C \exp\left(-(\eta \log n)^{\varrho/(\varrho-1)}\right) < C_{\gamma'} u^{(1-\varepsilon)\gamma'/\alpha},
\quad 
\text{for all 
$u \in \left[\tfrac{C}{C_{\gamma'}} \exp\left(- \tfrac{\alpha}{\gamma'} (\eta\log n)^\frac{\varrho}{(\varrho-1)}\right), c\right]$},
\]
and that $\e^{- \frac{\alpha}{\gamma'} (\eta\log n)^{\frac{\varrho}{\varrho-1}}} \ll 1/n$. Therefore, taking $\gamma = (1-\varepsilon)\gamma'< \alpha+\theta$ there exists $D_{\gamma}$ such that $\P\left( \Sigma_n \leq u \log n \right) \leq D_{\gamma} u^\frac{\gamma}{\alpha}$, for all $n$ large enough and $u \in [ \tfrac{1}{n},+\infty)$.

On the other hand if $u \leq 1/n$, let $j^* \in \N$ be such that $(1-\alpha)j^* > \gamma$, then
\[
\P(\Sigma_n < u \log n) = \P\bigg( \sum_{j=1}^n V_j^\alpha \leq u \log n \bigg)\leq \P \left( V_j^\alpha < u\log n ;\; \text{for all } 1 \leq j \leq j^* \right).
\]
Observe that if $u<1/n$ and $V_j^\alpha < u \log n$ for all $j \leq j^*$, we have
\[
 Y_j^\alpha = \frac{V_j^\alpha}{\left((1 - Y_1)(1-Y_2)\cdots (1-Y_{j-1})\right)^\alpha}
 \leq \frac{u \log n}{(1-Y_1^\alpha)\cdots (1-Y_{j-1}^\alpha)}.
\]
We prove by recurrence that under the above hypothesis
$Y_j^\alpha \leq u \log n\big/ \big(1 - \frac{(j-1) \log n}{n}\big)$. The case $j=1$ holds by the assumption
$Y_1^\alpha \leq u \log n$. 
Assuming that the statement holds for all $i \leq j-1$, that is 
$Y_i^\alpha \leq u \log n\big/ \big(1 - \frac{(i-1) \log n}{n} \big)$ then 
\begin{align*}
 (1 - Y_1^\alpha)(1-Y_2^\alpha) \cdots (1 - Y_{j-1}^\alpha) &\geq \prod_{i=1}^{j-1} \left(1 - \frac{\frac{\log n}{n}}{1 - \frac{(i-1)\log n}{n}} \right)\\  
 &\geq \prod_{i=1}^{j-1} \frac{1 - \frac{i \log n}{n}}{1- \frac{(i-1)\log n}{n}} = 1 - \frac{(j-1)\log n}{n},
\end{align*}
yielding $Y_j^\alpha \leq \log n\big/ \big(1 - \frac{(j-1) \log n}{n} \big)$. As a consequence, for all $j \leq j^*$ and $n$ sufficiently large
\[
Y_j^\alpha < 2u \log n
\quad \text{and hence} \quad
\P(\Sigma_n < u \log n) \leq \prod_{j=1}^{j^*} \P \left( Y_j^\alpha < 2 u\log n \right).
\]
Using crude estimate for the probability distribution function of the Beta distribution, we bound the product in the display by $C u^{\frac{\gamma}{\alpha}} \varrho \big(u^{(j^{*} (1-\alpha) - \gamma} (\log n)^{j^{*} (1-\alpha)}\big)$, with $C$ an explicit constant. Since $u<1/n$ and $j^{*} (1-\alpha) - \gamma>0$, the term inside the parentheses tends to zero uniformly in $u$. 
Therefore, increasing $D_\gamma>0$ if necessary, the upper-bound 
$\P\left( \Sigma_n < u \log n\right) \leq D_\gamma u^{\frac{\gamma}{\alpha}}$ holds for all $n \geq 1$ and $u \geq 0$ finishing the proof.
\end{proof}

In some cases, we are able to identify the random variable $\Psi_\alpha M^\alpha_\infty$.
\begin{corollary}
\label{cor:identification}
Let $(U_n)_n$ be a PD$(\alpha,0)$, then $\Psi_\alpha M^\alpha_\infty=L^{-\alpha}$, where $1/L=\lim_{n \to +\infty} n^{1/\alpha} U_n$.
\end{corollary}

\begin{proof}
By Proposition~\ref{prop:PD-PPP}, $L := \lim_{n \to +\infty} n^{-1/\alpha} / U_n$ exists a.s. and by Lemma~\ref{lem:mainbege} we have that
\begin{equation*}
 \Psi_\alpha M_\infty^\alpha \sim \frac{1}{\log n} \sum_{j=1}^n V_j^\alpha \leq \frac{1}{\log n} \sum_{j=1}^n U_j^\alpha \sim L^{-\alpha} 
 \qquad \text{as $n \to +\infty$,}
\end{equation*}
thus $\Psi_\alpha M_\infty^\alpha \leq L^{-\alpha}$ a.s. By Lemma \ref{lem:martingale}, the $p$th moments of $\Psi_\alpha M_\infty^\alpha$ are equal to 
\[
 \E\left[ \left(\Psi_\alpha M_\infty^\alpha \right)^p \right] = \frac{\Gamma(p+1)}{\Gamma(p\alpha + 1)}\Gamma(1-\alpha)^{-p},
 \quad
 \text{ for all $p > -1$.}
\]
By \cite[Equation (30)]{PiY97}, it matches with the $p$th moments of $L^{-\alpha}$, which implies that the two random variables have the same distribution (the Mittag-Leffler ($\alpha$) distribution), and hence that $\Psi_\alpha M_\infty^\alpha = L^{-\alpha}$ a.s. by monotonicity.
\end{proof}

\section{Convergence of discrete exchangeable coalescent processes}
\label{sec:cvexchangeable}
In this section, we study a family of coalescent processes with dynamics driven by PD-distributions and obtain a sufficient criterion for the convergence in distribution of these processes.
For the sake of completeness, we include a brief introduction to coalescent theory with the main results we will use, for a detailed account we recommend \cite{Bere09} from where we borrow the approach. 

Let $\mathcal{P}_n$ be the set of partitions (or equivalence classes) of $[n] := \{ 1,\ldots, n \}$ and $\mathcal{P}_\infty $ the set of partitions of $\N= [\infty]$. A partition $\pi \in \mathcal{P}_n$ is represented by blocks $\pi(1), \pi(2), \ldots$ listed in the increasing
order of their least elements, that is, $\pi(1)$ is the block (class) containing $1$, $\pi(2)$ the block containing the smallest element not in $\pi(1)$ and so on.
There is a natural action of the symmetric group $S_n$ on $\mathcal{P}_n$ setting 
$\pi^\sigma := \big\{ \{ \sigma(j), j \in \pi(i) \}, i \in [n] \big\}$ for $\sigma \in S_n$.
If $m<n$, one can define the projection of $\mathcal{P}_n$ onto $\mathcal{P}_m$ by the restriction $\pi|_m = \{\pi(j)\cap [m]\}$. 
For $\pi,\pi' \in \mathcal{P}_n$, we define the \textit{coagulation of $\pi$ by $\pi'$} to be the partition 
$
\mathrm{Coag}(\pi,\pi') = \left\{ \cup_{i \in \pi'(j)} \pi(i) ;\; j \in \N \right\}.
$ 

With this notation, a coalescent process $\Pi(t)$ is a discrete (or continuous) time Markov process in $\mathcal{P}_n$ such that for any $s,t \geq 0$,
\[
 \Pi(t+s) = \mathrm{Coag}(\Pi(t),\tilde{\Pi}_s), \quad \text{with } \tilde{\Pi}_s \text{ independent of } \Pi(t).
\]
We say that $\Pi(t)$ is \textit{exchangeable} if $\Pi^\sigma(t)$ and $\Pi(t)$ have the same distribution for all permutation~$\sigma$. 

An important class of continuous-time exchangeable coalescent processes in $\mathcal{P}_\infty$ are the so-called $\Lambda$-coalescents \cite{Pit99}, introduced independently by Pitman and Sagitov.
They are constructed as follows: let $\Pi_{n}(t)$ be the restriction of $\Pi(t)$ to $[n]$, then $(\Pi_n(t); t \geq 0)$ is a Markov jump process on $\mathcal{P}_n$ with the property
that whenever there are $b$ blocks, each $k$-tuple ($k\geq 2$) of blocks is merging to form a single block at the rate
\[
 \lambda_{b,k} = \int_0^1 x^{k-2} (1-x)^{b-k}\Lambda(\mathrm{d}x) , \quad \text{where $\Lambda$ is a finite measure on $[0,1]$.}
\]
Among such, we distinguish the Beta$(2-\lambda,\lambda)$-coalescents obtained from $\Lambda(\mathrm{d}x) = \frac{x^{1-\lambda} (1-x)^{\lambda-1}}{\Gamma(\lambda)\Gamma(2-\lambda)} \mathrm{d}x$, where $\lambda \in (0,2)$, the case $\lambda= 1$ (uniform measure) being the celebrated \emph{Bolthausen-Sznitman coalescent}.

The set $\mathcal{P}_\infty$ can be endowed with a topology making it a Polish space, therefore, one can study the weak convergence of processes in $\mathcal{D}\big( [0,\infty), \mathcal{P}_\infty \big)$, see \cite{Bere09} for the definitions. Without going into details, we say that a process $\Pi^N(t) \in \mathcal{P}_\infty$ converges in the \emph{Skorokhod sense} (or in distribution) to $\Pi(t)$, if for all $n \in \N$ the projection $\Pi^N(t)|_n$ converges in distribution to $\Pi(t) |_n$ in $\mathcal{D}\big( [0,\infty) \mathcal{P}_n \big)$.


\subsection{Coalescent processes obtained from multinomial distributions}
\label{subsec:coalescentProcess}

In this section, we define a family of discrete-time coalescent processes $(\Pi^N(t) ; t \in \N)$ and prove  sufficient criteria for its convergence in distribution. Let $(\eta^N_1, \ldots \eta^N_N)$ be an $N$-dimensional random vector satisfying
\[
 1 \geq \eta^N_1 \geq \eta^N_2 \geq \cdots \geq \eta^N_N \geq 0 \quad \text{and} \quad \sum_{j=1}^N \eta^N_j = 1.
\]
Conditionally on a realization of $(\eta^N_j)$, let $\big\{\xi_j; j \leq N \big\}$ be i.i.d. random variables satisfying $\P(\xi_j = k| \eta^N) = \eta^N_k$ and define the partition $\pi_N= \big\{ \{ j \leq N : \xi_j = k \} ;\; k \leq N \big\}$. With $(\pi_t;\; t \in \N)$ i.i.d. copies of $\pi_N$, let $\Pi^N(t)$ be the discrete time coalescent such that
\[
 \Pi^N(0) = \left\{ \{1\}, \{2\}, \ldots , \{n\} \right\} \quad \text{and} \quad \Pi^N(t+1) = \mathrm{Coag}\left( \Pi^N(t), \pi_{t+1} \right).
\]
The goal of this section is to obtain conditions under which $\Pi^N(t)$ converges in distribution. First, we assume that there exist a sequence $L_N$ and a function $f : (0,1) \to \R_+$ such that
\begin{equation}
 \label{eqn:convergence}
 \lim_{N \to +\infty} L_N = +\infty, \quad 
 \lim_{N \to +\infty} L_N \P\left(\eta^N_1 > x \right) = f(x) 
 \quad \text{and} \quad \lim_{N \to +\infty} L_N \E\left( \eta^N_2 \right) = 0.
\end{equation}
Denote by $c_N = \sum_{j=1}^N \E\big[(\eta_j^N )^2\big]$, which corresponds to the probability that two individuals have a common ancestor one generation backward in time.

\begin{lemma}\label{lem:criteria:lambda:coalescent} 
Assume that \eqref{eqn:convergence} holds and that
\begin{equation}
 \label{eqn:criteria:lambda:coalescent}
 \int_0^1 x \left( \sup_{N \in \N} L_N \P(\eta^N_1 > x) \right) dx < +\infty.
\end{equation}
Then, $c_N \sim_{N\to\infty} L_N^{-1} \int_0^1 2 x f(x) \mathrm{d}x$ and the re-scaled coalescent process $\big(\Pi^N (t/c_N); t \in \R_+ \big)$ converges in distribution to the $\Lambda$-coalescent, with $\Lambda$ satisfying $\int_x^1 \frac{\Lambda(\mathrm{d}y)}{y^2} = f(x)$.
\end{lemma}

\begin{proof}
Denote by $\nu_k = \#\{ j \leq N : \xi_j = k\}$, then $(\nu_1, \ldots, \nu_N)$ has multinomial distribution with $N$ trials and (random) probabilities outcomes $\eta_i^N$. By \cite[Theorem 2.1]{MoS01}, the convergence of finite dimensional distribution of $\Pi^N(t)$ is obtained from the convergence of the factorial moments of $\nu$, that is
\[ 
\frac{1}{c_N (N)_b}\sum_{\substack{i_1, \ldots, i_a =1 \\ \text{all distinct}}}^{N} \E \Big[ (\nu_{i_1})_{b_1} \ldots (\nu_{i_a})_{b_a} \Big], \quad \text{with } b_i \geq 2 \text{ and } b= b_1+\ldots+b_a,
\]
where $(n)_a := n (n-1) \ldots (n-a+1)$.
Since $(\nu_1, \ldots, \nu_N)$ is multinomial distributed, we obtain that
$\E \left[ (\nu_{i_1})_{b_1} \ldots (\nu_{i_a})_{b_a} \right] = (N)_b \E \left[ \eta_{i_1}^{b_1} \ldots \eta_{i_a}^{b_a} \right]$, see \cite[Lemma~4.1]{Cor14} for a rigorous a proof.
Therefore, we only have to show that for all $b$ and $a \geq 2$ 
\[
 \lim_{N \to +\infty} c_N^{-1} \sum_{i_1=1}^{N} \E \Big[ (\eta^N_{i_1})^{b} \Big] = \int_0^1 x^{b-2} \Lambda(\mathrm{d}x)
 \quad \text{and} \quad
 \lim_{N \to +\infty} c_N^{-1} \sum_{\substack{i_1, \ldots, i_a =1 \\ \text{all distinct}}}^{N}\E \Big[ (\eta^N_{i_1})^{b_1} \ldots (\eta^N_{i_a})^{b_a} \Big] = 0.
\]
We obtain by dominated convergence that 
\[
L_N \E \left[ \left(\eta^N_1 \right)^{2} \right] = \int_{0}^1 2 x L_N \P \left(\eta^N_1 > x \right) \mathrm{d}x \to \int_{0}^1 2 x f(x) \mathrm{d}x = \int_0^1 \Lambda(\mathrm{d}x) < +\infty,
\quad \text{as $N\to \infty$.}
\] 
Since $\eta^{N}_i$ are ordered and sum up to $1$, we also get 
$ \E \big[\big(\eta^{N}_2\big)^2 + \ldots + \big(\eta^{N}_N\big)^2 \big] \leq \E \big[\eta^{N}_2 (1-\eta_1)\big]$, 
with $L_N \E \left(\eta^N_2\right)$ tending to zero as $N \to \infty$. In particular, it implies that 
$L_N c_N = L_N\sum \E[(\eta^N_{i})^2]$ tends to $\int_0^1 2 x f(x) \mathrm{d}x$ as $N \to \infty$.
A similar calculation shows that for any $b \geq 2$ 
\[
 \lim_{N\to \infty} L_N \sum_{i=1}^{N} \E \big[ (\eta^N_{i})^{b}\big] = \int b x^{b-1} f(x) \mathrm{d}x = \int x^{b-2} \Lambda (\mathrm{d}x) = \lambda_{b,b},
\]
where $\lambda_{b,b}$ is the rate at which $b$ blocks merge into one given that there are $b$ blocks in total. The others $\lambda_{b,k}$ can be easily obtained using the recursion formula
$\lambda_{b,k} = \lambda_{b+1,k}+\lambda_{b+1,k+1}$.

We now consider the case $a=2$, cases $a>2$ being treated in the same way. We have
\begin{align*}
 &\sum_{\substack{i_1,i_2 =1 \\ \text{distinct}}}^{N} \E \Big[ \big(\eta^N_{i_1}\big)^{b_1} \big(\eta^N_{i_2}\big)^{b_2} \Big]\\
 \leq &\E \bigg[ \big(\eta^N_{1}\big)^{b_1} \eta^N_{2} \sum \big(\eta^N_i\big)^{b_2-1} + \big(\eta^N_1\big)^{b_2} \eta^N_2 \sum \big(\eta^N_i\big)^{b_1-1} + \sum_{i_1 \neq 1} \big(\eta^N_{i_1}\big)^{b_1} \eta^N_2 \sum_{\substack{i_2 \neq 1 \\ i_2 \neq i_1}} \big(\eta^N_{i_2}\big)^{b_2-1} \bigg] \\
 \leq &3 \times \E \big[ \eta^N_2 \big],
\end{align*}
we recall that $\big(\eta^N_2\big)^b + \ldots +\big(\eta^N_N\big)^b \leq \eta^N_2 + \ldots + \eta^N_N = 1 -\eta^N_1 < 1$. Since $L_N \E \eta^N_2 \to 0$ as $N\to \infty$, the right hand side of the inequality tends to zero, concluding the proof.
\end{proof}

The next lemma gives sufficient conditions for the convergence to the Kingman's coalescent.
\begin{lemma}
\label{lem:converngence:kingman} 
Assume that \eqref{eqn:convergence} holds and that
\begin{equation}
 \label{eqn:criteria:kingman:coalescent}
 \int_0^1 x f(x) dx = +\infty \quad \text{and} \quad \exists n \geq 2 : \int_0^1 x^n \left( \sup_{N \in \N} L_N \P(\eta^N_1 > x) \right) dx < +\infty. 
\end{equation}
Then, $\lim_{N\to+\infty}c_N L_N= +\infty$ and the ancestral partition process $(\Pi^N_n ( \lfloor t c_N^{-1} \rfloor); t \in \R_+)$ converges in the Skorokhod sense to the Kingman's coalescent restricted to $\mathcal{P}_n$.
\end{lemma}

\begin{proof}
A similar argument to the one used in Lemma~\ref{lem:criteria:lambda:coalescent} shows that
\[
L_N c_N \geq L_N \E[\eta_1^2] = \int_{0}^{1} 2 x L_N \P(\eta_1 > x) \mathrm{d}x.
\]
Thus, Fatou lemma and (\ref{eqn:criteria:kingman:coalescent}) yield $\liminf_{N \to +\infty} L_N c_N = +\infty$, proving the first claim. By \cite[Theorem 2.5]{Bere09}, the convergence to the Kingman's coalescent follows from 
\[
\textstyle
\lim_{N\to \infty} \sum_{i=1}^N \E[(\nu_i)_3] \big/(N)_3 c_N =0.
\] 
We rewrite the sum in the display as $c_N^{-1}\sum_{i=1}^N \E[(\eta^N_i)^3] $ and apply H\"older inequality to obtain 
\[
 \E\left[ (\eta^N_1)^\lambda (\eta^N_1)^{3-\lambda} \right] \leq \E\left[ (\eta^N_1)^2 \right]^{\lambda/2} \E\left[ (\eta^N_1)^\frac{2(3-\lambda)}{2-\lambda} \right]^{1-\lambda/2}
 \qquad \text{for all $\lambda <2$.}
\]
Let $\lambda \in (0,2)$ be the unique solution of $2(3-\lambda)\big/(2-\lambda) = n+4$, then we obtain from \eqref{eqn:criteria:kingman:coalescent} that 
\[
 \frac{\E\left[ (\eta^N_1)^3 \right]}{c_N} \leq \frac{\E\left[ (\eta^N_1)^3 \right]}{\E\left[ (\eta^N_1)^2 \right]} \leq \left(\frac{\E\left[ (\eta^N_1)^{n+1} \right]}{\E\left[ (\eta^N_1)^2 \right]} \right)^{1-\lambda/2} \underset{N \to +\infty}{\longrightarrow} 0.
\]
A similar argument to the one used in Lemma~\ref{lem:criteria:lambda:coalescent} shows that 
$c_N^{-1}\sum_{i=2}^N \E[(\eta^N_i)^3]$ tends to zero as $N\to \infty$, which proves the statement.
\end{proof}

\subsection{The Poisson-Dirichlet distribution case}
\label{subsec:PD_coalescent}

In this section, we construct a coalescent using the PD distribution and obtain a criterion for its convergence in distribution. 
With $(V_j, j \geq 1)$ a size-biased pick from a PD$(\alpha,\theta)$ partition, define 
\[
 \theta^N_j := \frac{V_j^\alpha}{\sum_{i=1}^N V_i^\alpha} 
 \qquad \text{and} \quad
\theta^N_{(1)} \geq \theta^N_{(2)} \geq \cdots \geq \theta^N_{(N)}, 
\quad \text{the order statistics of $(\theta^N_j)$.}
\]
In what follows, $\theta^N_{(i)}$ will stand for the $\eta_i^N$ from Section~\ref{subsec:coalescentProcess} and $(\Pi^N_n (t); t \in \N)$ for the coalescent with transition probabilities $ \Pi^N_n (t+1) = \mathrm{Coag}\left( \Pi^N_n (t) , \pi_{t}^n\right)$, as defined there. 

\begin{theorem}
\label{thm:mainbege}
With the above notation, set $\lambda = 1 + \theta/\alpha$ and
\[L_N = c_{\alpha,\theta} (\log N)^{\lambda}, \quad \text{where} \quad c_{\alpha,\theta} = \left(\Gamma(1-\theta/\alpha)\Gamma(1-\alpha)^{\theta/\alpha} \Gamma(1+\theta)\right)^{-1}. \]
\begin{enumerate}
\item \label{thm:mainbege.beta} If $\theta \in (-\alpha,\alpha)$, then $c_N \sim_{N \to +\infty} (1-\theta/\alpha)/L_N$ and $(\Pi^N(t/c_N), t \geq 0)$ converges weakly to the Beta$(2-\lambda,\lambda)$-coalescent. 
\item \label{thm:mainbege.kingman} Otherwise, $\lim_{N \to +\infty} c_NL_N = +\infty$ and $(\Pi^N(t/c_N))$ converges weakly to the Kingman's coalescent.
\end{enumerate}
\end{theorem}

Before proving Theorem~\ref{thm:mainbege} we obtain a couple of technical results. The next lemma studies the asymptotic behavior of $\theta^N_1$.

\begin{lemma}
\label{lem:thetaestimate}
With the notations of Theorem~\ref{thm:mainbege}, we have
\[
 \lim_{N \to +\infty} L_N \P\left( \theta^N_1 > x \right) = \frac{1}{\lambda\Gamma(\lambda)\Gamma(2-\lambda)} \left(\frac{1-x}{x}\right)^\lambda = \int_x^1 \frac{\mathrm{Beta}(2-\lambda,\lambda)(dy)}{y^2}.
\]
Moreover, there exists $C>0$ such that for all $x \in (0,1)$, $\displaystyle \sup_{N \in \N} L_N\P\left( \theta^N_1 > x \right) \leq Cx^{-\lambda}$.
\end{lemma}

\begin{proof}
Let $\Sigma'_N := \sum_{j=2}^N \left(\frac{V_j}{1-Y_1}\right)^{\alpha}$, then Remark \ref{rem:changeofparameter} says that $\Sigma_N'$ and $Y_1$ are independent and that $\Sigma_N'$ has the distribution of $V_1' + \ldots V_{N-1}'$ with $V_i'$ a size-biased pick from a PD$(\alpha,\alpha+\theta)$ distribution. By Lemma~\ref{lem:mainbege}, for all $\epsilon \in (0,1)$ there exists $C= C(\epsilon)$ and $N_0 \in \N$ such that
\begin{equation} 
 \label{eqn:thetamajoration}
 \sup_{N \geq N_0} \P\left( \Sigma'_N \leq u \log N \right) \leq \min\left( C u^{\lambda+\epsilon}, 1 \right),
 \qquad \text{for all $u \geq 0$}.
\end{equation}
Writing $\P\left( \theta^N_1 > x \right)$ in terms of $\Sigma_N'$ and $V_1 =Y_1$, we obtain
\begin{align*}
 \P\left( \theta^N_1 > x \right) &= \P\left( V_1^{\alpha} > x \left( V_1^{\alpha} + (1-V_1)^{\alpha} \Sigma'_N \right) \right)
 = \P\left( \tfrac{V_1}{1-V_1} > \left(\tfrac{x}{1-x}\Sigma'_N\right)^{1/\alpha} \right)\\
 &= \int_0^1 \P\left( 1/y - 1 > \left(\tfrac{x}{1-x}\Sigma'_N\right)^{1/\alpha}\right) \frac{\Gamma(1+\theta)(1-y)^{-\alpha}y^{\alpha + \theta-1}}{\Gamma(1-\alpha)\Gamma(\alpha+\theta)} dy.
\end{align*}
Making the change of variables $u=(\frac{1-x}{x\log N}) (1/y-1)^{\alpha}$ the display reads
\[
 \P\left( \theta^N_1 > x \right)
 = \left(\frac{1-x}{x \log N}\right)^\lambda \frac{\Gamma(1+\theta)}{\alpha \Gamma(1-\alpha)\Gamma(\alpha + \theta)} \int_0^{+\infty} \frac{\P\left( \Sigma'_N < u \log N\right)}{u^{2-1/\alpha} \left(u^{1/\alpha} + \left( \frac{1-x}{x \log N} \right)^{1/\alpha} \right)^{1+\theta}}du.\]
Then, we use \eqref{eqn:thetamajoration} to bound the equation within the integral, obtaining
\[
 \frac{\P\left( \Sigma'_N < u \log N\right)}{u^{2-1/\alpha} \left(u^{1/\alpha} + \left( \frac{1-x}{x \log N} \right)^{1/\alpha} \right)^{1+\theta}}
 \leq \frac{\P(\Sigma'_N \leq u \log N)}{u^{2+\theta/\alpha}} \leq \min(Cu^{\epsilon-1}, u^{-2}),
\]
for all $N$ large enough . In particular, there exists $C>0$ such that $L_N \P\left( \eta^N_1 > x \right) \leq C x^{-\lambda}$ for all $N \in \N$. Moreover, by dominated convergence and Lemma \ref{lem:mainbege}, we obtain
\begin{align*}
 \lim_{N \to +\infty}(\log N)^\lambda\P\left( \theta^N_1 > x \right)
 & =\left(\frac{1-x}{x}\right)^\lambda \frac{\Gamma(1+\theta)}{\alpha \Gamma(1-\alpha)\Gamma(\alpha + \theta)} \int_0^{+\infty} \frac{\P(\Psi_\alpha (M'_\infty)^\alpha < u)}{u^{1 + \lambda}}\\
 & = \left(\frac{1-x}{x}\right)^\lambda \frac{\Gamma(1+\theta)}{\alpha \lambda \Gamma(1-\alpha)\Gamma(\alpha +\theta)} \E\left( \left( \Psi_\alpha (M'_\infty)^\alpha \right)^{-\lambda} \right)\\
 & = \left(\frac{1-x}{x}\right)^\lambda \frac{\alpha^{\alpha+\theta-1} \Gamma(1-\alpha)^{\theta/\alpha} \Gamma(1+\theta)}{\lambda \Gamma(\alpha+\theta)} \Phi_{\theta+\alpha,\alpha}(-(\theta+\alpha)),
\end{align*}
and hence $\displaystyle \lim_{N \to +\infty}L_N\P\left( \theta^N_1 > x \right) = \left(\frac{1-x}{x}\right)^\lambda \frac{1}{\lambda\Gamma(\lambda)\Gamma(2-\lambda)}$, proving the statement.
\end{proof}

This result is used to study the asymptotic behavior of $\theta^N_{(1)} = \max_{j \leq N} \theta^N_j$.
\begin{lemma}
\label{lem:firstMaximum}
For all $\epsilon \in (0,1)$, there exists $C= C(\epsilon)$ such that 
\[
 \left| \P\big( \theta^N_{(1)} > x \big) - \P\left( \theta^N_1 > x \right) \right| \leq C(x \log N)^{\epsilon-2-\theta/\alpha},
 \qquad \text{for all $x \in (0,1)$ and $N$ large enough}.
\]
\end{lemma}

\begin{proof}
Notice that $\P\left( \theta^N_1 > x \right) \leq \P\big( \theta^N_{(1)} > x \big)$ and that $\theta^N_{(1)} = \theta^N_{1}$ if $V_1 >1/2$, thanks to $\sum V_i \equiv 1$. Therefore, splitting the events according to $V_1> 1/2$ and $V_1<1/2$ we obtain
\[
  \P\big( \theta^N_{(1)} > x \big) - \P\left( \theta^N_1 > x \right)
  = \! \P\left( \theta^N_{(1)} > x; V_1 \leq \frac{1}{2} \right)\! -
   \! \P\left( \theta^N_1 > x; V_1 \leq \frac{1}{2}  \right) \!
  \leq \P\left( \theta^N_{(1)}>x; \leq \frac{1}{2} \right).
\]
Since $0 < V_j < 1$ and $V_1 = Y_1$, we have
\begin{align*}
 \P\left( \theta^N_{(1)}>x, V_1 \leq 1/2 \right)
 &= \textstyle \P\left( \max_{j \leq N} V_j^{\alpha} > x \sum_{j=1}^N V_j^{\alpha} ; V_1 \leq 1/2 \right)\\
 &\leq \P\left( x^{-1}> Y_1^{\alpha} + (1-Y_1)^{\alpha} \Sigma'_N, Y_1 \leq 1/2 \right) \leq \P\left( \Sigma'_N < 2^\alpha/ x\right),
\end{align*}
where $\Sigma'_N = \sum V_j^{\alpha}/(1-Y_1)^{\alpha}$. By Lemma \ref{lem:mainbege}, we obtain that for all $N$ sufficiently large
\[
 \P\big( \theta^N_{(1)}>x, V_1 \leq 1/2 \big) \leq C(x \log N)^{\epsilon-1-\frac{\alpha+\theta}{\alpha}},
\]
which finishes the proof.
\end{proof}

We now study the asymptotic behavior of $\theta^N_{(2)}$, the second maxima in $\{\theta_1^N, \ldots, \theta_N^N \}$.
\begin{lemma}
\label{lem:secondMaximum}
For all $\epsilon \in (0,1]$, there exists $C>0$ such that for any $x \in (0,1)$ and $N \in \N$,
\[
 \P\left( \theta^N_{(2)} > x \right) < C(x \log N)^{\epsilon-2-\theta/\alpha}.
\]
\end{lemma}

\begin{proof}
We basically use the same method as in the previous lemma
\begin{align*}
 \P & \big( \theta^N_{(2)} > x \big) \\
 & = \P\big( \theta^N_{(2)} > x, V_1 \leq 1/2 \big) + \P\big( \theta^N_{(2)}>x, V_1 > 1/2, V_2 < 1/3 \big) + \P\big( \theta^N_{(2)} > x, V_1 > 1/2, V_2>1/3 \big)\\
 &\leq \P\big( \theta^N_{(1)} > x, V_1 \leq 1/2 \big) + \P\big( \theta^N_{(2)} > x, V_1 > 1/2, V_2 < 1/3 \big) + \P\big( \theta^N_2 > x \big).
\end{align*}
By Lemma~\ref{lem:firstMaximum}, we have that $\P\big( \theta^N_{(1)} > x, V_1 \leq 1/2 \big) \leq C (x \log N)^{\epsilon-2-\theta/\alpha}$, so the same arguments used in Lemma \ref{lem:thetaestimate} yield
\[
 \P\big( \theta^N_2 > x \big) 
 = \P\big( V_2^{\alpha}(1-x) - x V_1^{\alpha} > x (1-V_1-V_2)^{\alpha} \Sigma''_N \big)
 \leq C (x \log N)^{\epsilon -2-\theta/\alpha},
\]
with $\Sigma''_N :=(1-V_1-V_2)^{-\alpha} \sum_{j=3}^N V_j^{\alpha}$. Moreover, $\Sigma''_N $ is independent of $(V_1,V_2)$ and 
\begin{align*}
\P\left( \theta^N_{(2)} > x, V_1 > 1/2, V_2 < 1/3 \right)
 & = \P\left( \max_{2 \leq j \leq N} V_j^{\alpha} > x \left(V_1^{\alpha} + V_2^{\alpha} + (1-V_1-V_2)^{\alpha} \Sigma''_N \right) \right)\\
 &\leq \P\left( \Sigma''_N \leq C/x \right) \leq C(x \log N)^{\epsilon-2-\theta/\alpha},\qquad \qquad
\end{align*}
concluding the proof.
\end{proof}

\begin{proof}[Proof of Theorem \ref{thm:mainbege}]
Given $\varepsilon > 0$, Lemma \ref{lem:firstMaximum} says that there exists $C>0$ such that 
\[
 L_N \P\big( \theta^N_{(1)} > x \big)- L_N \P\left( \theta^N_1 > x\right) \leq C(\log N)^{\epsilon - 1} x^{-\lambda},
 \quad \text{for all $x \in (0,1)$}.
\]
Therefore, by Lemma \ref{lem:thetaestimate} we have that
\begin{equation}
 \label{eqn:theta1}
 \begin{split}
\lim_{N \to +\infty} L_N \P\big( \theta^N_{(1)} > x \big) = \frac{1}{\lambda \Gamma(\lambda)\Gamma(2-\lambda)}\left(\frac{1-x}{x}\right)^{\lambda} & \quad \text{and}\\
 & \sup_{N \in \N} L_N \P\big( \theta^N_{(1)} > x \big) \leq C x^{-\lambda}.
 \end{split}
\end{equation}
We obtain from Lemma \ref{lem:secondMaximum} that
\begin{equation}
 \label{eqn:theta2}
 L_N \E\left[ \theta^N_{(2)} \right] = \int_0^1 L_N \P\big( \theta^N_{(2)}>x \big) \dx \leq (\log N)^{\epsilon - 1} \int_0^1 x^{\epsilon - 1 - \lambda} \dx \underset{N \to +\infty}{\longrightarrow} 0,
\end{equation}
which implies that $\Pi^N(t)$ satisfies \eqref{eqn:convergence}.

Assume now that $\theta \in (-\alpha,\alpha)$, so that $\lambda \in (0,2)$, then \eqref{eqn:theta1} yields
\[
 \int_0^1 x \sup_{N \in \N} L_N \P\big( \theta^N_{(1)} > x \big) \dx \leq C \int_0^1 x^{1-\lambda} \dx < +\infty.
\]
Therefore, the assumptions of Lemma~\ref{lem:criteria:lambda:coalescent} are satisfied implying that $\Pi^N(t/c_N)$ converges in distribution to the Beta$(2-\lambda,\lambda)$-coalescent and that $c_N L_N \sim (1 - \theta/\alpha)$ as $N\to \infty$.
On the other hand if $\theta \geq \alpha$, we have that $\int_0^1 x \left(\frac{1-x}{x}\right)^\lambda \dx = +\infty$. With $k \geq \lambda$, we obtain from (\ref{eqn:theta1}) that 
\[
 \int_0^1 x^k \sup_{N \in \N} L_N \P\big( \theta^N_{(1)} > x \big) \dx \leq C \int_0^1 x^{k-\lambda} \dx < +\infty.
\]
We apply Lemma \ref{lem:converngence:kingman} to conclude that $\Pi^N(t/c_N)$ converges weakly to the Kingman's coalescent.
\end{proof}

\section{Poisson-Dirichlet representation of the \texorpdfstring{$(N,\beta)$}{(N,B)}-branching random walk}
\label{sec:piy}

In this section, we explore the relations between the $(N,\beta)$-BRW and the PD$(\beta^{-1},0)$ distribution to show Theorems~\ref{thm:main_vel} and~\ref{thm:main_coal} in the case where $\beta < \infty$. The case $\beta = \infty$ is also studied, but using different methods.

\begin{proposition}
\label{prop:def2}
With $\beta \in (1, \infty)$, let $(U_n)_n$ be a PD$(\beta^{-1},0)$, $L = \lim_{n \to +\infty} n^{-\beta} U_n^{-1}$ and $(V_n)_n$ be its size-biased pick, then
\[
 \left(X^N_1(j)-X^N_0(\eq), j \leq N\right) \egaldistr \left( \tfrac{1}{\beta} \log V_j + \tfrac{1}{\beta} \log L \right) ,
\]
in particular, $ X^N_1(\eq)-X^N_0(\eq) \egaldistr \log \sum_{j=1}^N V_j^{1/\beta} + \frac{1}{\beta} \log L.$
\end{proposition}

\begin{proof}
By Proposition~\ref{prop:def}, $(x_k, k \geq 1 ) := \mathrm{Rank}\left(\left\{ X^N_0(j) + p - X^N_0(\eq), p \in \mathcal{P}_1(j), j \leq N \right\} \right)$ is the ordered points of a PPP$(\e^{-x} \dx)$. With
$L = \sum_{j=1}^{+\infty} \e^{\beta x_j}$ and $\quad U_j = \e^{\beta x_j}/L$, we know from 
Proposition~\ref{prop:PD-PPP} that $(U_j, j \geq 1)$ is a PD$(\beta^{-1},0)$ and that $\lim_{n \to +\infty} n^\beta U_n = L^{-1}$. 
By the definition of the $(N,\beta)$-BRW, $V_j := \e^{\beta (X^N_1(j)-X^N_0(\eq))}/L$ is the $j$th particle sampled in the size-biased pick from $(U_n)_n$. Inversing the equation, we conclude that
\[
 X^N_1(j) - X^N_0(\eq) = \tfrac{1}{\beta}\left(\log V_j + \log L\right),
\]
proving the first statement. The second statement follows from the definition of $X_1^N(\eq)$. 
\end{proof}

To study the case $\beta = +\infty$, we use the following representation of $N$ rightmost points of a PPP($\e^{-x}\dx$).
\begin{proposition}
\label{prop:betainfini}
Let $x_1 > x_2 >\ldots > x_N$ the $N$ rightmost points of a PPP($\e^{-x}\dx$), then
\[
 ( x_1, \ldots, x_N ) \egaldistr \mathrm{R}\mathrm{ank}\{ Z_N+e_1, \ldots, Z_N + e_N \},
\]
where $(e_j)$ are i.i.d exponential random variables with mean $1$ and $Z_N$ is an independent random variable satisfying $\P(Z_N \in \dx) = \frac{1}{N!}\exp(-(N+1)x - \e^{-x}) \dx$.
\end{proposition}

\begin{proof}
It is an elementary result about PPP that 
\[\textstyle
\sum_{j=1}^{+\infty} \delta_{\e^{-x_j}} \egaldistr \text{PPP($\dx$) on $\R_+$}
\quad \text{and that} \quad 
\e^{-x_{N+1}}\egaldistr \text{Gamma$(N+1,1)$},
\] 
moreover, conditionally on $x_{N+1}$, $(\e^{-x_{1}}, \ldots \e^{-x_n})$ are the ranked values of $N$ i.i.d. uniform random variables on the interval $[0,\e^{-x_{N+1}}]$. Setting $Z_N= x_{N+1}$ and $U_1, \ldots, U_N$ i.i.d. uniform random variables 
\[
 (\e^{-x_1}, \ldots, \e^{-x_N}) \egaldistr \mathrm{R}\mathrm{ank}\{\e^{-Z_{N}} U_1, \ldots, \e^{-Z_{N}} U_N \}.
\]
It is straightforward that $Z_N$ and $(x_1, \ldots, x_N)$ satisfy the desired properties. 
\end{proof}

We first use these results to compute the asymptotic behavior of the speed of the $(N,\beta)$-BRW. 

\begin{proof}[Proof of Theorem \ref{thm:main_vel}]
Lemma~\ref{lem:speed} says that \eqref{equa:lem:speed} holds with $v_{N,\beta} = \E\left( X^N_1(\eq)-X^N_0(\eq)\right)$. Thus, if $\beta < \infty$ Proposition~\ref{prop:def2} yields
\[
 v_{N,\beta} =\E\left( X^N_1(\eq)-X^N_0(\eq)\right) = \E\left( \log \left( \sum_{j=1}^N V_j^{1/\beta} \right) \right) + \frac{1}{\beta} \E\left( \log L \right).
\]
By Lemma \ref{lem:mainbege}, $(\log N)^{-1}\sum_{j=1}^N V_j^{1/\beta}$ converges to $\Psi_{\beta^{-1}} M_\infty^{1/\beta}$ a.s. and in $\mathbb{L}^1$ as $N\to \infty$. Therefore, the logarithm of this quantity converges a.s. as well. We notice from Lemma~\ref{lem:mainbege} that 
\[
 \P\left(\log\left( (\log N)^{-1}\sum_{j=1}^N V_j^{1/\beta} \right) \leq -u \right) = \P\left( \sum_{j=1}^N V_j^{1/\beta} \leq \e^{-u} \log N\right) \leq D_{2/\beta} \e^{-2u},
 \quad \text{for all $u>0$}.
\]
The $\mathbb{L}^1$ convergence of $(\log N)^{-1}\sum_{j=1}^N V_j^{1/\beta}$ implies the existence of a constant $K$ such that
\[
 \P\left(\log\left( (\log N)^{-1}\sum_{j=1}^N V_j^{1/\beta} \right) \geq u \right) \leq \P\left( \sum_{j=1}^N V_j^{1/\beta} \geq \e^{u} \log N \right) \leq K \e^{-u}
\quad \text{for all } u \geq 0.
\]
In particular, $(\log \sum_{j=1}^N V_j^{1/\beta} -\log \log N)$ is uniformly integrable, which implies its $\mathbb{L}^1$ convergence. 
We know from Corollary \ref{cor:identification} that $\Psi_{\beta^{-1}} M_\infty^{1/\beta} = L^{1/\beta}$, and hence that 
\begin{align*}
 \lim_{N\to \infty} v_{N,\beta} - \log \log N
 &= \lim_{N\to \infty} \E\left[ \log \frac{ \sum_{j=1}^N V_j^{1/\beta}}{\log N} \right] + \frac{\E\left[ \log L \right]}{\beta}\\ 
 &= \E\left[ \log\left(\Psi_{\beta^{-1}} M_\infty^{1/\beta}\right) \right] + \frac{\E\left[ \log L \right]}{\beta} = 0.
\end{align*}

For $\beta = \infty$ we follow the ideas from \cite{BDMM07} and use Laplace methods to estimate the asymptotic mean of $X^N_1(\eq)$. Let $(x_1,\ldots x_N)$ be the $N$ largest atoms of a PPP($\e^{-x}\dx$), and define for $\lambda > 0$ 
\[
 \Lambda(\lambda) := \E\left( \exp\left(-\lambda \log \sum_{k=1}^N \e^{x_k} \right) \right) = \E\left(\left( \sum_{k=1}^N \e^{x_k}\right)^{-\lambda}\right).
\]
By Proposition \ref{prop:betainfini}, we have
$
\Lambda(\lambda) = \E\left(\e^{-\lambda Z}\right) \E \left( \left( \sum_{k=1}^N \e^{e_k} \right)^{-\lambda}\right),
$
where $(e_k)$ are i.i.d exponential random variables and $\exp (-Z_N)$ has Gamma$(N+1,1)$ distribution. Notice that the following equalities hold: $\E\left(\e^{-\lambda Z}\right) = \frac{\Gamma(N+1+\lambda)}{\Gamma(N+1)}$ and
\begin{equation}
 \E \left( \left( \sum_{k=1}^N \e^{e_k} \right)^{-\lambda}\right)
 = \frac{1}{\Gamma(\lambda)}\int_0^{+\infty} t^{\lambda-1} \E\left( \e^{-t \sum_{k=1}^N \e^{e_k}} \right) \d t
 = \frac{1}{\Gamma(\lambda)}\int_0^{+\infty} t^{\lambda-1} I_0(t)^N \d t \label{eqn:momentFacile},
\end{equation}
with $I_0(t) = \E(\e^{-t\e^{e_1}})$ the Laplace transform of $\e^{e_1}$.
The function $I_0$ can be represented using the exponential integral $\mathrm{Ei}$
we have $I_0(x) = x \mathrm{Ei}(-x) + e^{-x}$. Therefore, there exists $K>0$ such that 
\[
 |I_0(x) - 1 - x\log x| \leq K x,
 \quad \text{for any $x \geq 0$.}
\]
In particular for $x = t/(N \log N)$ we have $|I_0(t/(N\log N)) - 1 - t/N| \leq \frac{K t}{N \log N}$. Thus, \eqref{eqn:momentFacile} yields
\begin{align*}
 \int_0^{+\infty} t^{\lambda-1} I_0(t)^N\d t
 &= \frac{1}{(N \log N)^\lambda} \int_0^{+\infty} t^{\lambda - 1} I_0(t/(N\log N))^N \d t\\
 &\leq \frac{1}{(N \log N)^\lambda} \int_0^{+\infty} t^{\lambda - 1} \e^{- t(1 - \frac{K}{\log N})} \d t \leq \frac{\Gamma(\lambda)}{(N \log N)^\lambda} (1 - \tfrac{K}{\log N})^{-\lambda}.
\end{align*}
The same argument with the obvious change gives a similar lower bound, which implies
\[
 \Lambda(\lambda) = \frac{\Gamma(N+1+\lambda)}{(N \log N)^\lambda\Gamma(N+1)}(1 + O((\log N)^{-1}) ),
\]
uniformly in $\lambda \in [0,1]$, therefore $\log \Lambda(\lambda) = \lambda \log \log N + O((\log N)^{-1})$. As a consequence
\[
 \E\left(\log \sum_{k=1}^N \e^{x_k}\right) = \lim_{\lambda \to 0} \frac{\log \Lambda(\lambda)}{\lambda} = \log \log N + o(1),
\]
which concludes the proof.
\end{proof}

In a similar way, we obtain the genealogy of the $(N,\beta)$-branching random walk.

\begin{proof}[Proof of Theorem \ref{thm:main_coal}] 
If $\beta \in (1,\infty)$, then Lemma \ref{lem:coal} and Proposition \ref{prop:def2} say that the genealogy of the $(N,\beta)$-BRW can be described by \ref{thm:mainbege.beta}. in Theorem \ref{thm:mainbege}, with $\alpha = \frac{1}{\beta}$ and $\theta = 0$. Therefore, it converges to the Bolthausen-Sznitman coalescent.

On the other hand if $\beta = \infty$, the genealogy of the $(N,\infty)$-BRW is again described by a coalescent process obtained from multinomial random variables. In this case, by Proposition~\ref{prop:betainfini} we can rewrite the coefficients $\eta^N_j$ as
\[
\eta^N_j = \frac{\e^{e_j}}{\sum_{i=1}^{N} \e^{e_i}}, \quad
 \text{with $e_1,\ldots, e_N$ i.i.d. exponential random variables}.
\] 
Thanks to $\P(\e^{e_j} \geq x) = x^{-1}$, \cite[Theorem 1.2 (c)]{Cor14} says that the genealogy of the  $(N,\infty)$-BRW converges to the Bolthausen-Sznitman coalescent with $c_N \sim N$ as $N\to \infty$.
\end{proof}

\label{Bibliography}
\bibliographystyle{plain}

\end{document}